\documentclass[10pt]{article}
\usepackage{amsmath}
\usepackage{color}
\usepackage{fancyhdr}
\usepackage{verbatim}
\usepackage{amsfonts,longtable,mathrsfs,mathtools}
\usepackage{epsfig}
\usepackage{amsfonts,mathtools, epsfig,float}
\usepackage{color}
\usepackage[numbers,sort&compress]{natbib}

\def\ec{\end{center}}
\def\bc{\begin{center}}
\def\ec{\end{center}}
\newtheorem{definition}{Definition}[section]
\newtheorem{theorem}{Theorem}[section]

\newtheorem{proof}{Proof}[section]
\begin{document}
\bc {\bf On a family  of higher order recurrence relations: symmetries, formula solutions, periodicity and stability analysis
  }\ec
\medskip
\bc
Mensah Folly-Gbetoula\footnote{Corresponding author:\\ Mensah.Folly-Gbetoula@wits.ac.za(M. Folly-Gbetoula)
	}
\vspace{1cm}
\\School of Mathematics, University of the Witwatersrand, Wits 2050, Johannesburg, South Africa.\\

\ec
\begin{abstract}
\noindent In this paper, we present formula solutions of a family of difference equations of higher order. We discuss the periodic nature of the solutions and we investigate the stability character of the equilibrium points. We utilize Lie symmetry analysis as part of our approach together with some number theoretic functions. Our findings generalize certain results in the literature. 
\noindent
\end{abstract}
\section{Introduction} \setcounter{equation}{0}
\noindent There has been great attention and focus on difference equations. Just like differential equations, there are techniques that one can use to solve difference equations. One method for solving them is to use Lie symmetry analysis. In this approach, one finds an invariant which can be used to find a simpler form of the equation. Amongst the first people to use Lie symmetry analysis to solve difference equations are Maeda \cite{Maeda} and Hydon \cite{hyd}. Examples of the use of difference equations in real-life include modeling the  proliferation of disease, loan payments, population studies, etc. \par \noindent
Recurrence equations of a general order have been investigated in the literature and from different approaches by several authors \cite{mmd, ladas, alma, beverton, banas, El, ndm, QM, DC, jocaa}.  In their work, Almatrafi and others \cite{alma} discussed the exact solutions, stability, oscillation and periodic aspects of the the difference equations
 \begin{equation}\label{eq1}
\eta _{n+1}=\frac{\eta_{n-4k+1}}{\pm 1 \pm \prod\limits_{i=1}^{k}
 	\eta_{n-4i+1}}.
 \end{equation}
These equations are special cases of a more generalized setting \begin{equation}\label{eq1g}
\eta _{n+1}=\frac{\eta_{n-4k+1}}{a_n +b_n \prod\limits_{i=1}^{k}
	\eta_{n-4i+1}},
\end{equation}
where $a_n$ and $b_n$ are arbitrary real sequences. One can readily see that with a suitable change of variables, the above  equations can be derived from the higher order Berverton-Holt difference \cite{beverton} equation 
\begin{equation}\label{app}
r_{n + l} = \frac{\mu_nK_nr_{n}}{K_n + (\mu_n-1)r_{n}}
\end{equation}
where $\mu_{n}>1$ denotes the growth rate, $K_n$ is the carrying capacity and $r_0, r_1, \ldots, r_{l-1}$ are the positive initial values. \\ \noindent
 In this paper, we perform a symmetry analysis of the equivalent equation
 \begin{equation}\label{eq1'}
 u_{n+4k}=\frac{u_n}{A_n+B_n\prod\limits_{i=1}^{k}
 	u_{n+4(i-1)}},
 \end{equation}
 for some arbitrary real sequences $A_n$ and $B_n$ where $u_0, u_1, \ldots, u_{4k-1}$ are initial values. Using symmetries, we obtain explicit formulas for the solution of \eqref{eq1'} and we deduce the solutions of \eqref{eq1g} from those of \eqref{eq1'}.  We also study the periodic nature of the solutions and a stability analysis of the difference equation is investigated. \par \noindent  
 The derivation of symmetries for higher order recurrences involves cumbersome calculations and, to the best of our knowledge, there are no computer packages that generate these symmetries. For more on Lie symmetry analysis of difference equations, the reader can refer to \cite{hyd} and, among others, the articles \cite{QR, md, mmd}.
\subsection{Preliminaries}
Consider the difference equation:
\begin{align}\label{eq2}
u_{n+4k}=\mathcal{G}(n,u_n,u_{n+4}, u_{n+8} \dots, u_{n+4k-4}),
\end{align}
 for some smooth function $\mathcal{G}$ satisfying ${\partial \mathcal{G}}/{\partial u_{n}} \neq 0$.  Symmetry groups are connected to the determination of infinitesimal transformations. Let
\begin{align}\label{gtrans}
\hat{u}_{n}=u_n+ \varepsilon Q(n,u_{n})+O(\varepsilon ^2),
\end{align}
be the one parameter Lie group of transformations of  \eqref{eq2} with the corresponding generator
\begin{align}\label{vector}
\mathcal{V}=Q(n,u_n)\frac{\partial\quad}{\partial u_n}.
\end{align}
Note that the knowledge of the characteristic $Q=Q(n,u_n)$ requires the knowledge of the $(4k-4)$-th prolongation of $\mathcal{V}$
\begin{align}\label{vectorp}
\mathcal{V}^{[4k-4]}=Q\frac{\partial\quad}{\partial u_{n}}+(S^4Q)\frac{\partial\quad}{\partial u_{n+4}}+\dots+(S^{4k-4}Q)\frac{\partial\quad}{\partial u_{n+4k-4}},
\end{align}
where $S^i:n\rightarrow n+i$ is the forward shift operator. It is known that \eqref{gtrans} is a symmetry group if and only if the condition  
 \begin{align}\label{LSC}
S^{4k}Q(n,u_{n})- \mathcal{V}^{[4k-4]}(\mathcal{G})=0 {\Big|}_{u_{n+5k}=\mathcal{G}(n,u_n,u_{n+4}, \dots, u_{n+4k-4})}
\end{align}
holds. Suppose that the characteristic is obtained by solving the functional equation \eqref{LSC}. One can use the canonical coordinate  \cite{JV}
 \begin{align}\label{cano}
c_n= \int{\frac{du_n}{Q(n,u_n)}}
\end{align}
to derive the invariants which may be used to lower the order of the difference equations. In \cite{hyd}, the author attests that with the choice of canonical coordinate \eqref{cano},  the recurrence equation can, without fail, be represented in the form
$c_{n+1}-c_n=d_n$
whose the solution takes the shape
 \begin{align}\label{redp}
c_n=  \sum_{k=n_0}^{n}d_k+w_1
\end{align}
for some constant $w_1$. From \eqref{redp}, it is not difficult to find the solution expressed in terms of the original variables. In this paper, our solution is obtained via the use of the canonical coordinate through a  different methodology. \par \noindent
The following theorem and definitions \cite{ladas} are useful for studying local and globally stability aspects of the equilibrium point. 
\begin{definition}
	The equilibrium point $\bar{u}$ of \eqref{eq2} is said to be locally stable if for any $\epsilon >0$ such that if $\{ u_n\}_{n=0}^{\infty}$ is a solution of \eqref{eq2} with 
	\begin{align}
	|u_0-\bar{u}|+ |u_1-\bar{u}|+\dots+ |u_{4k-2}-\bar{u}|+ |u_{4k-1}-\bar{u}|<\delta,
	\end{align}
	then 
		\begin{align}
		|u_n-\bar{u}|<\epsilon \quad \text {for all}    \quad n\geq 0.
		\end{align}
\end{definition}
\begin{definition}
	The equilibrium point $\bar{u}$ of \eqref{eq2} is said to be a global attractor if for any solution  $\{ u_n\}_{n=0}^{\infty}$ of \eqref{eq2},  
	\begin{align}
\lim_{n \leftarrow \infty } u_n =\bar{u}.
	\end{align}
\end{definition}
\begin{definition}
	The equilibrium point $\bar{u}$ of \eqref{eq2} is globally asymptotically stable if $\bar{u}$ is locally stable and is a global attractor of \eqref{eq2}.
\end{definition}
Let 
\begin{align}
p_i =\frac{\partial f}{\partial u_{n+i}}(\bar{u}, \dots, \bar{u}) , \quad i=4r, \quad r=0,1,\dots, k-1.
\end{align}
It follows that 
\begin{align}\label{ceq}
\lambda^{4k}-p_{4k-4}\lambda^{4k-4}-\dots-p_{4}\lambda^{4}-p_0=0
\end{align}
is the corresponding characteristic equation of \eqref{eq2} about the equilibrium point $\bar{u}$.
\begin{theorem}
Suppose $f$ is a smooth function defined on some open neighborhood of $\bar{u}$. Then the following statements are true:
\begin{itemize}
	\item[(i)] The equilibrium point $\bar{u}$ is locally asymptotically stable if all the roots of $ceq$ have absolute value less than one.
	\item[(ii)] The equilibrium point $\bar{u}$ is unstable if at least one root of $ceq$ has absolute value greater than one.
\end{itemize}
\end{theorem}
\begin{definition}
	The equilibrium point $\bar{u}$ of \eqref{eq2} is called non-hyperbolic if there exists a root of \eqref{ceq} with absolute value equal to one.
\end{definition}
\section{Lie analysis and solutions}
To derive the characteristic function $Q$ admitted by \eqref{eq1'}, that is, 
 \begin{align}\label{xn}
u_{n+4k}=\mathcal{G}=\frac{u_n}{A_n+B_n\prod\limits_{i=1}^{k}
	u_{n+4(i-1)}},
\end{align}
we apply the symmetry constraint equation  \eqref{LSC} to \eqref{xn} to get
\begin{align}\label{a1}
 &Q(n+5k, u_{n+5k})-\sum_{i=0}^{k-1}  \mathcal{G}_ {,u_{n+4i}}Q(n+4i, u_{n+4i})=0,
\end{align}
where $f_{,x}$ denotes the partial derivative of $f$ with respect to $x$.
To solve for $Q$, we first apply the differential operator {\scriptsize ${\partial}/{\partial u_{n+4}}-({u_{n+8}}/{u_{n+4}}){\partial}/{\partial u_{n+8}}$} on \eqref{a1}. This gives{\scriptsize
\begin{align}
&\left(\frac{u_{n+8}}{u_{n+4}}\mathcal{G}_{,u_{n+8}u_{n+4}}-\mathcal{G}_{,u_{n+4}u_{n+4}}\right)Q(n+4, u_{n+4})+
\left(\frac{u_{n+8}}{u_{n+4}}\mathcal{G}_{,u_{n+8}u_{n+8}}-\mathcal{G}_{,u_{n+4}u_{n+8}}\right)Q(n+8, u_{n+8})\nonumber\\&
+
\sum_{i\geq 3} \left(\frac{u_{n+8}}{u_{n+4}}\mathcal{G}_{,u_{n+8}u_{n+4i}}\right. 
 -\mathcal{G}_{,u_{n+4}u_{n+4i}}\bigg) Q(n+4i, u_{n+4i})-
\frac{B_nu_n^2\prod\limits_{i=3}^{k}
	u_{n+4(i-1)}}{\left(A_n+B_n\;\;\mathclap{\prod\limits_{i=1}^{k}}\;\;
	u_{n+4(i-1)}\right)^2}\big(Q'(n+8, u_{n+8})\nonumber\\&-Q'(n+4,u_{n+4})\big)+\left( \frac{u_{n+8}}{u_{n+4}}\mathcal{G}_{,u_{n+8}u_{n}}-\mathcal{G}_{,u_{n+4}u_{n}}\right)Q(n, u_n)=0
\end{align}}
which simplifies to {\scriptsize
\begin{align}
&u_{n+8}Q(n+4, u_{n+4})-u_{n+4}Q(n+8, u_{n+8})+u_{n+4}u_{n+8}(Q'(n+8,u_{n+8})-Q'(n+4,u_{n+4}))=0.
\end{align}}
We then differentiate the above equation with respect to $u_{n+4}$ twice to get
 \begin{align}\label{diffQ}
\left( u_{n+4}Q'(n+4,u_{n+4})-Q(n+4, u_{n+4})\right)''=0.
 \end{align}
The general solution of \eqref{diffQ} takes the form  
\begin{align}\label{solQ}
Q(n,u_{n})=\alpha _nu_n+\beta _n (u_n\ln u_n +u_n) +\gamma _n
\end{align} 
for some functions $\alpha _n$, $\beta _n$ and $\gamma _n$ of $n$. Next, we substitute $ Q$ and its corresponding shifts in \eqref{a1}.  Bearing in mind that $\alpha _n$, $\beta _n$ and $\gamma _n$ are independent of $u_{n}$ and their shifts, we use the method of separation. It turns out that $\beta_n$ is equal to zero and the system of overdetermined equations resulting from the separation is as follows: 
\begin{align}
\label{overdetermining}
1&:  A_n^2\gamma_{n+4k} -A_n \gamma_n =0\nonumber\\
u_n&:A_n\alpha_{n+4k}-A_n\alpha_n =0\nonumber\\
u_n^2u_{n+4}\dots u_{n+4k-4}&: B_n(\alpha_{n+4k}+\alpha_{n+4}+\alpha_{n+8}+\dots+\alpha_{n+4k-8}+\alpha_{n+4k-4})=0\nonumber\\
  u_nu_{n+4}\dots u_{n+4k-4}&:2A_nB_n\gamma_{n+4k}=0
\end{align}
which reduces to 
\begin{align}
&\gamma_n =0, \quad \alpha_{n+4k}-\alpha_n =0, \\ & \alpha_{n}+\alpha_{n+4}+\alpha_{n+8}+\dots+\alpha_{n+4k-8}+\alpha_{n+4k-4}=0.\label{constraint}
\end{align}
Equation \eqref{constraint} is a linear difference equation with constant coefficients and has the characteristic equation 
 \begin{align}\label{constraintchara} 1+r^{4}+r^{8}+\dots+r^{4k-8}+r^{4k-4}=0.
 \end{align}
It is well known that if $r$ is a solution of \eqref{constraintchara}, then $\alpha_n = r^n$ is a solution of \eqref{constraint}. Multiplying \eqref{constraintchara} by $1-r^4$ and solving the resulting equation gives 
\begin{align}
r_1(s)=e^{\frac{is\pi}{2k}},\quad r_2(s)=-e^{\frac{is\pi}{2k}},\quad r_3(s)=ie^{\frac{is\pi}{2k}},\quad r_4(s)=-ie^{\frac{is\pi}{2k}}, 
\end{align}
for $ 1\leq s\leq k-1.$ Thus, from \eqref{solQ}, the finite dimensional Lie algebra is spanned by the vectors fields
{\scriptsize \begin{align}\label{gener}
&X_{1}(s)=e^{\frac{ins\pi}{2k}} u_{n}{\frac{\partial\quad} {\partial u_{n}}},\;X_{2}(s)=(-1)^ne^{\frac{ins\pi}{2k}} u_{n}{\frac{\partial\quad} {\partial u_{n}}},\\
& X_{3}(s)=i^ne^{\frac{ins\pi}{2k}} u_{n}{\frac{\partial\quad} {\partial u_{n}}},\;X_{4}(s)=(-i)^ne^{\frac{ins\pi}{2k}} u_{n}{\frac{\partial\quad} {\partial u_{n}}}
\end{align}}
for $ 1\leq s\leq k-1.$
To obtain the compatible variable, we use the  canonical coordinate
\begin{equation}\label{a7}
C(n)= \int\frac{du_n}{\alpha_n u_{n}}=\frac{1}{\alpha_n}\ln|u_n|,
\end{equation}
 where $\alpha_n$ satisfies \eqref{constraint}. Replacing $\alpha_n$ with $C(n)\alpha_n$ in the left hand side of equation in \eqref{constraint} yields the group invariant
$I_n=\ln|u_nu_{n+4}\dots u_{n+4k-4}|$
since $(X_r(s)) I_n =0$ for $r=1, 2, 3, 4$. Hence, using \eqref{xn}, we have
$e^{-I_{n+4}}={A_ne^{-I_n}}+ B_n$.
For the sake of simplicity, we instead use the invariant 
\begin{align}\label{rn}
r_n=\exp{(-I_n)}=\frac{1}{u_nu_{n+4}\dots u_{n+4k-4}}.
\end{align}
On one hand, shifting \eqref{rn} four times and replacing $u_{n+4k}$ in the resulting equation yields
 \begin{align}
 \label{a9g}
r_{n+4}=A_n{r_n}+ A_n
\end{align}
whose iteration gives
  \begin{align}
  \label{rnsol}
  r_{4n+j}=r_j\left( \prod\limits_{k_1=0}^{n-1}A_{4k_1+j}\right) + \sum\limits_{l=0}^{n-1}\left(B_{4l+j}\prod\limits_{k_2=l+1}^{n-1} A_{4k_2+j}\right)
  \end{align}
for $ j=0,1,2,3$. On the other hand, using the same relation given in \eqref{rn}, we have 
 \begin{align}\label{unrn}
 u_{n+4k}= \frac{r_n}{r_{n+4}}u_n.
\end{align}
We iterate \eqref{unrn} and its solution in closed form takes the form
\begin{align}\label{unrn'}
u_{4kn+i}= u_i \left( \prod\limits_{s=0}^{n-1}\frac{r_{4ks+i}}{r_{4ks+4+i}}\right), \quad i=0,\dots, 4k-1.
\end{align}
From the known fact that any integer $r$ can be written as $r=4\lfloor r/4\rfloor+\tau(r), \; 0\leq \tau(r)\leq 3$, where $\tau(r)$ is the remainder when $r$ is divided by $4$, we can rewrite \eqref{unrn'} as follows:
 \begin{align}\label{unrnsol}
 u_{4kn+i}= u_i \left( \prod\limits_{s=0}^{n-1}\frac{r_{4(ks+\lfloor \frac{i}{4}\rfloor)+\tau(i)}}{r_{{4(ks+\lfloor \frac{i+4}{4}\rfloor)+\tau(i+4)}}}\right),
 \end{align}
where $i=0,\dots, 4k-1$ and $0\leq \tau(i)\leq 3$. Using \eqref{unrn'} in \eqref{unrnsol}, we obtain {\scriptsize
  \begin{align}\label{unsol'}
  u_{4kn+i}=& u_i  \prod\limits_{s=0}^{n-1}\frac{r_{\tau(i)}\left( \prod\limits_{k_1=0}^{\substack{ks+\\ \lfloor \frac{i}{4}\rfloor-1}}A_{4k_1+\tau(i)}\right) + \sum\limits_{l=0}^{\substack{ks+\\\lfloor \frac{i}{4}\rfloor-1}}\left(B_{4l+\tau(i)}\prod\limits_{k_2=l+1}^{\substack{ks\\+\lfloor \frac{i}{4}\rfloor-1}} A_{4k_2+\tau(i)}\right)}{r_{\tau(i+4)}\left( \prod\limits_{k_1=0}^{\substack{ks+\\\lfloor \frac{i+4}{4}\rfloor-1}}A_{4k_1+\tau(i+4)}\right) + \sum\limits_{l=0}^{\substack{ks+\\\lfloor \frac{i+4}{4}\rfloor-1}}\left(B_{4l+\tau(i+4)}\prod\limits_{k_2=l+1}^{\substack{ks+\\\lfloor \frac{i+4}{4}\rfloor-1}} A_{4k_2+\tau(i+4)}\right)},
\end{align}}
$i=0,\dots, 4k-1$. Noting that $\tau(i+4)=\tau(i)$ and $1/r_i=\prod_{j=0}^{k-1}{u_{i+4j}}$, the above equation simplifies
{\scriptsize \begin{align}\label{unsol}
 u_{4kn+i}=&  u_i  \prod\limits_{s=0}^{n-1}\frac{\left( \prod\limits_{k_1=0}^{\substack{ks+\\ \lfloor \frac{i}{4}\rfloor-1}}A_{4k_1+\tau(i)}\right) + \left(\prod\limits_{j=0}^{k-1}{u_{\tau(i)+4j}}\right)\sum\limits_{l=0}^{\substack{ks+\\\lfloor \frac{i}{4}\rfloor-1}}\left(B_{4l+\tau(i)}\prod\limits_{k_2=l+1}^{\substack{ks\\+\lfloor \frac{i}{4}\rfloor-1}} A_{4k_2+\tau(i)}\right)}{\left( \prod\limits_{k_1=0}^{\substack{ks+\lfloor \frac{i}{4}\rfloor}}A_{4k_1+\tau(i)}\right) + \left(\prod\limits_{j=0}^{k-1}{u_{\tau(i)+4j}}\right)\sum\limits_{l=0}^{\substack{ks+\lfloor \frac{i}{4}\rfloor}}\left(B_{4l+\tau(i)}\prod\limits_{k_2=l+1}^{\substack{ks+\lfloor \frac{i}{4}\rfloor}} A_{4k_2+\tau(i)}\right)},
 \end{align} }
$i=0,\dots, 4k-1$. The solution of \eqref{eq1g} is obtained by back shifting \eqref{unsol} $4k-1$ times. Hence,  the closed form solution of \eqref{eq1g} is given by
 {\scriptsize
 	\begin{align}\label{xnsol}
 	\eta_{4kn-4k+1+i}=&\eta_{i-4k+1}\;\; \mathclap{\prod\limits_{s=0}^{n-1}}\;\;\frac{\left(\quad \mathclap{\prod\limits_{k_1=0}^{\substack{ks+\\ \lfloor \frac{i}{4}\rfloor-1}}}\;\;a_{4k_1+\tau(i)}\right) + \left(\prod\limits_{j=0}^{k-1}{\eta_{\tau(i)-4k+1+4j}}\right)\sum\limits_{l=0}^{\substack{ks+\\\lfloor \frac{i}{4}\rfloor-1}}\left(b_{4l+\tau(i)}\;\;\mathclap{\prod\limits_{k_2=l+1}^{\substack{ks+\\\lfloor \frac{i}{4}\rfloor-1}}}\;\; a_{4k_2+\tau(i)}\right)}{\left(\;\quad \mathclap{\prod\limits_{k_1=0}^{\substack{ks+\lfloor \frac{i}{4}\rfloor}}}\;\;a_{4k_1+\tau(i)}\right) + \left(\prod\limits_{j=0}^{k-1}{\eta_{\tau(i)-4k+1+4j}}\right)\sum\limits_{l=0}^{\substack{ks+\lfloor \frac{i}{4}\rfloor}}\left(b_{4l+\tau(i)}\;\;\mathclap{\prod\limits_{k_2=l+1}^{\substack{ks+\lfloor \frac{i}{4}\rfloor}}}\quad a_{4k_2+\tau(i)}\right)}.
 	\end{align}} \noindent
 Observe that if $\{a_n\}$ and $\{b_n\}$ are constant sequences, i.e. $a_n = a$ for all $n$ and $b_n = b$ for all $n$, then the solutions of \eqref{eq1g} and \eqref{eq1'} are given by
   	\begin{align}\label{xnsolcst}
   	\eta_{4kn-4k+1+i}	=& \eta_{i-4k+1}  \prod\limits_{s=0}^{n-1}\frac{ a^{ks+ \lfloor \frac{i}{4}\rfloor} + b\left(\prod\limits_{j=0}^{k-1}{\eta_{\tau(i)-4k+1+4j}}\right)\sum\limits_{l=0}^{\substack{ks+\\\lfloor \frac{i}{4}\rfloor-1}}a^l}{a^{ks+ \lfloor \frac{i}{4}\rfloor+1}  +b \left(\prod\limits_{j=0}^{k-1}{\eta_{\tau(i)-4k+1+4j}}\right)\sum\limits_{l=0}^{\substack{ks+\lfloor \frac{i}{4}\rfloor}}a^l}
   	\end{align}
   	and 
 \begin{align}\label{unsolcst}
 u_{4kn+i}=& u_i   \prod\limits_{s=0}^{n-1}\frac{A^{ks+ \lfloor \frac{i}{4}\rfloor} + B\left(\prod\limits_{j=0}^{k-1}{u_{\tau(i)+4j}}\right)\sum\limits_{l=0}^{\substack{ks+\\\lfloor \frac{i}{4}\rfloor-1}}A^l}{A^{ks+ \lfloor \frac{i}{4}\rfloor+1}  + B\left(\prod\limits_{j=0}^{k-1}{u_{\tau(i)+4j}}\right)\sum\limits_{l=0}^{\substack{ks+\\\lfloor \frac{i}{4}\rfloor}}A^l},
 \end{align}
respectively.\par \noindent In the following section, we investigate some special cases. One of the aims  is to realize some results in \cite{alma}.
\section{Special cases}
 \subsection{The case when $a = 1$ and $b$ is a  constant}We investigate the case when $a = 1$ and $b$ constant. In this case, from \eqref{xnsolcst}, the solution is given by 
 	\begin{align}\label{xnsolcsta1}
 	\eta_{4kn-4k+1+i}	=& \eta_{i-4k+1} \prod\limits_{s=0}^{n-1}\frac{1 + b\left(\prod\limits_{j=0}^{k-1}{\eta_{\tau(i)-4k+1+4j}}\right)(ks+\lfloor \frac{i}{4}\rfloor)}{1  +b \left(\prod\limits_{j=0}^{k-1}{\eta_{\tau(i)-4k+1+4j}}\right)(ks+\lfloor \frac{i}{4}\rfloor+1)},
 	\end{align}
$i=0,\dots, 4k-1$. Replacing $b=\pm 1$ in \eqref{xnsolcsta1} yields the results in \cite{alma} (see Theorems 1 and 6). In fact, noting that $\tau(4k-1-j)=3-\tau(j)=3-j+4\lfloor j/4\rfloor$, we have:
  	\begin{align}\label{theo1and6}
  	\eta_{4kn-j}=&	\eta_{4kn-4k+1+(4k-1-j)},\qquad j=0, 1, \dots, 4k-1\\
  		=& \eta_{-j} \prod\limits_{s=0}^{n-1}\frac{1 + b\left(\prod\limits_{r=0}^{k-1}{\eta_{\tau(4k-1-j)-4k+1+4r}}\right)(ks+\lfloor \frac{4k-1-j}{4}\rfloor)}{1  +b \left(\prod\limits_{r=0}^{k-1}{\eta_{\tau(4k-1-j)-4k+1+4r}}\right)(ks+\lfloor \frac{4k-1-j}{4}\rfloor+1)}\\
  		=& \eta_{-j} \prod\limits_{s=0}^{n-1}\frac{1 + b\left(\prod\limits_{r=0}^{k-1}{\eta_{-j+4\lfloor \frac{j}{4}\rfloor-4r}}\right)(ks+k-1-\lfloor \frac{j}{4}\rfloor)}{1  +b \left(\prod\limits_{r=0}^{k-1}{\eta_{-j+4\lfloor \frac{j}{4}\rfloor-4r}}\right)(ks+k-\lfloor \frac{j}{4}\rfloor)}.
\end{align}
  \subsection{The case when $a \neq 1$ and $b$ is a  constant}Here, from \eqref{xnsolcst}, the solution is given by 
  {\scriptsize
  	\begin{align}\label{xnsolcstanot1'}
  	\eta_{4kn-4k+1+i}	=& \eta_{i-4k+1}  \prod\limits_{s=0}^{n-1}\frac{a^{ks+\lfloor \frac{i}{4}\rfloor} + b\left(\prod\limits_{j=0}^{k-1}{\eta_{\tau(i)-4k+1+4j}}\right)(\frac{1-a^{ks+\lfloor \frac{i}{4}\rfloor}}{1-a})}{a^{ks+\lfloor \frac{i}{4}\rfloor+1}  +b \left(\prod\limits_{j=0}^{k-1}{\eta_{\tau(i)-4k+1+4j}}\right)(\frac{1-a^{ks+\lfloor \frac{i}{4}\rfloor+1}}{1-a})}
  	\end{align}}
and, similarly, this can also be written in the form
 {\scriptsize 	\begin{align}\label{xnsolcstanot1}
  	\eta_{4kn-j}	=& \eta_{-j}  \prod\limits_{s=0}^{n-1}\frac{a^{ks+k-1-\lfloor \frac{i}{4}\rfloor} + b\left(\prod\limits_{j=0}^{k-1}{\eta_{-j+4\lfloor \frac{j}{4}\rfloor-4r}}\right)(\frac{1-a^{ks+k-1+\lfloor \frac{i}{4}\rfloor}}{1-a})}{a^{ks+k-\lfloor \frac{i}{4}\rfloor}  +b \left(\prod\limits_{j=0}^{k-1}{\eta_{-j+4\lfloor \frac{j}{4}\rfloor-4r}}\right)(\frac{1-a^{ks+k-\lfloor \frac{i}{4}\rfloor}}{1-a})}.
  	\end{align}}
For $a=-1$, the above equation simplifies to
\begin{itemize} 
	\item For $k$ even,
	 		\begin{align}\label{xnsolcsta1keven}
	 	\eta_{4kn-j}	=&  \eta_{-j}  \prod\limits_{s=0}^{n-1}\frac{-(-1)^{\lfloor \frac{i}{4}\rfloor} + b\left(\prod\limits_{j=0}^{k-1}{\eta_{-j+4\lfloor \frac{j}{4}\rfloor-4r}}\right)(\frac{1+(-1)^{\lfloor \frac{i}{4}\rfloor}}{2})}{(-1)^{\lfloor \frac{i}{4}\rfloor}  +b \left(\prod\limits_{j=0}^{k-1}{\eta_{-j+4\lfloor \frac{j}{4}\rfloor-4r}}\right)(\frac{1-(-1)^{\lfloor \frac{i}{4}\rfloor}}{2})}\\
	 	=& \eta_{-j}\left[-1+b\left(\prod\limits_{j=0}^{k-1}{\eta_{-j+4\lfloor \frac{j}{4}\rfloor-4r}}\right)\right]^{(-1)^{\lfloor \frac{j}{4}\rfloor}n}.
	 	\end{align}
	 This result was obtained in \cite{alma} for $b=\pm 1$ (see Theorems 11 and 18).
		\item For $k$ odd,
		 		\begin{align}\label{xnsolcsta1odd}
		 	\eta_{4kn-j}	=& \eta_{-j}  \prod\limits_{s=0}^{n-1}\frac{(-1)^{s-\lfloor \frac{i}{4}\rfloor} + b\left(\prod\limits_{j=0}^{k-1}{\eta_{-j+4\lfloor \frac{j}{4}\rfloor-4r}}\right)(\frac{1-(-1)^{s-\lfloor \frac{i}{4}\rfloor}}{2})}{-(-1)^{s-\lfloor \frac{i}{4}\rfloor}  +b \left(\prod\limits_{j=0}^{k-1}{\eta_{-j+4\lfloor \frac{j}{4}\rfloor-4r}}\right)(\frac{1+(-1)^{s-\lfloor \frac{i}{4}\rfloor}}{2})}\\
		 		=& \begin{cases}
		 		\eta_{-j}, \quad \text{if  $n$ is even } \\ \\
		 		\eta_{-j}\left[{-1+b\left(\prod\limits_{r=0}^{k-1}{\eta_{-j+4\lfloor \frac{j}{4}\rfloor-4r}}\right)}\right]^{(-1)^{\lfloor \frac{j}{4}\rfloor+1}}, \text{ if $n$ is odd}
		 		\end{cases}
		 	\end{align}
for $j=0, 1, \dots, 4k-1$ and furthermore, 
		  		\begin{align}\label{xnsolcsta1odd8k}
		  	\eta_{8kn-j}	=& \eta_{-j}
		  	\end{align}
for all $n$. More explicitly, the $8k$ periodic solutions are as follows:
{\scriptsize
	 \begin{align}
\eta_{i} =& \eta_{4k(1)-(4k-i)}=\frac{\eta_{-4k+i}}{\left[{-1+b\left(\prod\limits_{r=0}^{k-1}{\eta_{-4(r+1)+i}}\right)}\right]}, i=1,2,3,4\\
\eta_{4+i} =& \eta_{4k(1)-(4k-i-4)}={\eta_{-4k+i+4}}{\left[{-1+b\left(\prod\limits_{r=0}^{k-1}{\eta_{-4(r+1)+i}}\right)}\right]}, i=1,2,3,4\\
\vdots & \qquad \qquad \qquad \vdots  \qquad \qquad \qquad \vdots \qquad \qquad \qquad \vdots\nonumber\end{align}\begin{align}
\eta_{4k-8+i} =& \eta_{4k(1)-(8-i)}={\eta_{8-i}}{\left[{-1+b\left(\prod\limits_{r=0}^{k-1}{\eta_{-4(r+1)+i}}\right)}\right]}, i=1,2,3,4\\
\eta_{4k-4+i} =& \eta_{4k(1)-(4k-i-4)}=\frac{\eta_{-4k+i+4}}{\left[{-1+b\left(\prod\limits_{r=0}^{k-1}{\eta_{-4(r+1)+i}}\right)}\right]},\quad i=1,2,3,4\\
\eta_{4k+1+i}=& \eta_{4k(2)-(4k-1-i)}=\eta_{-4k+1+i},\quad i=0, 1, \dots, 4k-1.
\end{align}}
For this special case, the results  were obtained in \cite{alma} for $b=\pm 1$ (see Theorems 9, 15 and 16).
	\end{itemize}
\section{Periodicity and behavior of the solutions}
 \begin{theorem}
 	Let $u_n$ be a solution of 
	\begin{equation}\label{eq1cstp}
 	u_{n+4k}=\frac{u_n}{A+B\prod\limits_{i=1}^{k}
 		u_{n+4(i-1)}},
\end{equation}
 	for some non-zero constants $A\neq 1$ and $B$. Suppose the initial conditions $x_i, \; i=0,\dots, 4k-1$, are such that $\prod_{j=0}^{k-1}
 	u_{4j+p}=(1-A)/B$, $p=0, 1, 2, 3$. Then the solution of \eqref{eq1cstp} is periodic with period $4k$.
 \end{theorem}
 \begin{proof}
 	Suppose $\prod\limits_{j=0}^{k-1} u_{4j+p}=(1-A)/B$, $p=0, 1, 2, 3$. From \eqref{unsolcst}, we get 
 	\begin{align}\label{unsolcst4p}
 	u_{4kn+i}=&  u_i  \prod\limits_{s=0}^{n-1}\frac{A^{ks+ \lfloor \frac{i}{4}\rfloor} +B \left(\prod\limits_{j=0}^{k-1}{u_{\tau(i)+4j}}\right)\sum\limits_{l=0}^{\substack{ks+\\\lfloor \frac{i}{4}\rfloor-1}}A^l}{A^{ks+ \lfloor \frac{i}{4}\rfloor+1} + B\left(\prod\limits_{j=0}^{k-1}{u_{\tau(i)+4j}}\right)\sum\limits_{l=0}^{\substack{ks+\lfloor \frac{i}{4}\rfloor}}A^l}\\
 	=&  u_i  \prod\limits_{s=0}^{n-1}\frac{A^{ks+ \lfloor \frac{i}{4}\rfloor} +B \left(\frac{1-A}{B}\right)\left(\frac{1-A^{ks+\lfloor \frac{i}{4}\rfloor}}{1-A}\right)}{A^{ks+ \lfloor \frac{i}{4}\rfloor+1} + B\left(\frac{1-A}{B}\right)\left(\frac{1-A^{ks+\lfloor \frac{i}{4}\rfloor+1}}{1-A}\right)}\\ =&  u_i,
 	\end{align} 
for all $i=0,1, \dots, 4k-1$ since $0\leq \tau(i)\leq 3$.
 \end{proof}
	\begin{figure}[ht]
 		\centering
		\includegraphics[scale=0.30]{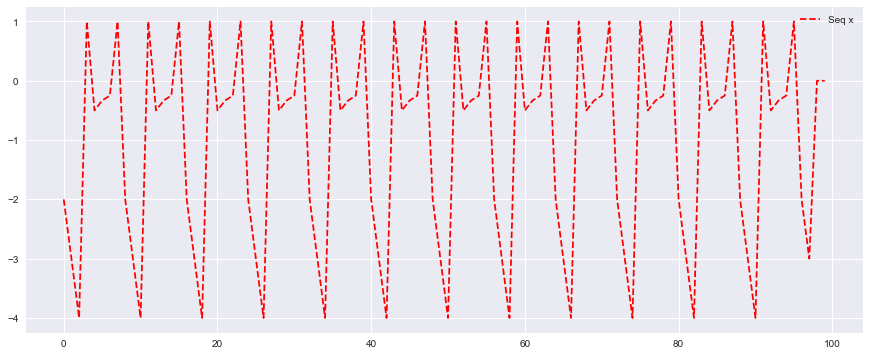}
 		\caption{Graph of 
 			$x_{n+8} = \dfrac{x_n}{(2-x_nx_{n+4})}$, where $x_0=-2,x_1=-3, x_2=-4, x_3=1, x_4=-1/2, x_5=-1/3, x_6=-1/4, x_7=1$  and are such that $x_0x_4=x_1x_5=x_2x_6=x_3x_7=(1-A)/B$.\label{M17_1}}
 	\end{figure}
We plot Figure~\ref{M17_1} to illustrate Theorem 4.1. We note that for $A=-1$ and $B=1$, we get the result in Theorem 13 in \cite{alma} and the result's restriction ($\prod_{r=0}^{k-1} \eta_{-j+4\lfloor \frac{j}{4}\rfloor-4r}=2, j=0, 1, \dots, 4k-1$ or simply $-j+4\lfloor \frac{j}{4}\rfloor=0, 1, 2, 3$) is a special case of the assumption in the above theorem ($\prod_{j=0}^{k-1} u_{4j+p}=(1-A)/B$, that is, $\prod_{r=0}^{k-1} \eta_{p-4r\lfloor \frac{j}{4}\rfloor-4r}=(1-A)/B, p=0, 1, 2, 3$). \par \noindent
Observe that in Theorem 15 in \cite{alma}, the authors ought to add the restriction $v_j\neq -2$. If this condition is not satisfied, the period will be $4k$ and not $8k$ as they clearly stated in Theorem 20 in \cite{alma}.
\begin{theorem}
Let $u_n$ be a solution of 
 	\begin{equation}\label{eq1cstnonp}
 	u_{n+4k}=\frac{u_n}{1+B\prod\limits_{i=1}^{k}
 		u_{n+4(i-1)}},
 	\end{equation}
 	for some non-zero constant $B$.  The zero equilibrium point is non hyperbolic. Furthermore, if the initial conditions $x_i, \; i=0,\dots, 4k-1$, and $B$ are positive, then the solution converges to the zero equilibrium point.
 	\end{theorem}
 	\begin{proof}
 		The equilibrium point of \eqref{eq1cstnonp} is $u=0$. Let {\scriptsize  \begin{align}
 			f(u_n, u_{n+4}, \dots, u_{n+4(k-1)})=	u_{n+4k}=\frac{u_n}{1+B\prod\limits_{i=1}^{k}
 				u_{n+4(i-1)}}.\end{align}} So,
 		{\scriptsize 	\begin{align}
 			f_{,u_n}=\frac{1}{(1+B\prod\limits_{i=1}^{k}
 				u_{n+4(i-1)})^2}, \quad
 			f_{,u_{n+4j}}=\frac{-Bu_n^2}{(1+B\prod\limits_{i=1}^{k}
 				u_{n+4(i-1)})^2} \prod\limits_{\substack{ i=1\\i\neq j}}^{k-1}
 			u_{n+4i},\quad j=1,2,\dots, k-1.
 			\end{align}}
 		We have $	f_{,u_n}(0,0,\dots,0)=1$ and $f_{,u_{n+4j}}(0,0,\dots,0)=0, j=1,2,\dots, k-1$. Thus, the characteristic equation associated with \eqref{eq1cstp} is $\lambda ^{4k}-1=0$ and therefore, $|\lambda_i|=1$. Therefore, the zero equilibrium point is non-hyperbolic.\par \noindent
 Suppose the non-zero initial conditions are all positive. From \eqref{unsolcst}, we get 
 		{\scriptsize 
 			\begin{align}\label{unsolcsta1}
 			u_{4kn+i}=&  u_i  \prod\limits_{s=0}^{n-1}\frac{1+B \left(\prod\limits_{j=0}^{k-1}{u_{\tau(i)+4j}}\right)(ks+\lfloor \frac{i}{4}\rfloor)}{1 + B\left(\prod\limits_{j=0}^{k-1}{u_{\tau(i)+4j}}\right)(ks+\lfloor \frac{i}{4}\rfloor+1)}\\
 			=&  u_i  \prod\limits_{s=0}^{n-1}\left( 1- \frac{B \left(\prod\limits_{j=0}^{k-1}{u_{\tau(i)+4j}}\right)}{1 + B\left(\prod\limits_{j=0}^{k-1}{u_{\tau(i)+4j}}\right)(ks+\lfloor \frac{i}{4}\rfloor+1)}\right)\\=& u_i  \prod\limits_{s=0}^{n-1} \Theta(s).\label{theta}\end{align}}
If $B$ is positive, $\Theta(s)<1$, $s=0, 1, \dots, n-1$. Therefore, $u_{n}$ tends to zero as $n$ tends to infinity.
 	\end{proof}
 		\begin{figure}[ht]
 			\centering
 	 		\includegraphics[scale=0.30]{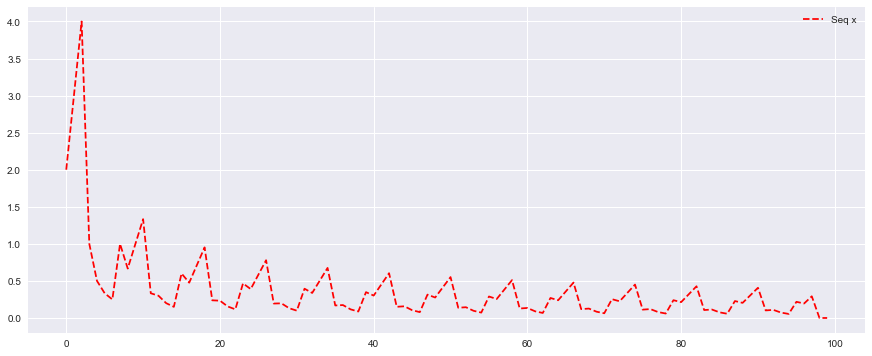}
 			\caption{Graph of 
 				$x_{n+8} = {x_n}/{(2+x_nx_{n+4})}$, where $x_0=2,x_1=3, x_2=4, x_3=1, x_4=1/2, x_5=1/3, x_6=1/4, x_7=1$.\label{M17_2}}
 		\end{figure}
 		We plot Figure~\ref{M17_2} to illustrate Theorem 4.2.
 	\begin{theorem}
 		Assume that $B$ is positive and $u_i\geq 0, i=0, 1, \dots, k-1$.Then the zero equilibrium point $u=0$ of \eqref{eq1cstnonp} is globally asymptotically stable.
 	\end{theorem}
 	\begin{proof}
 	The equilibrium point of  \eqref{eq1cstnonp} satisfies $u(1+Bu^k)=0$. Thus, $u=0$. Let $\epsilon \geq 0$ and suppose the $u_i's , i=0, 1, \dots, k-1$ are such that $$|u_i|\leq \frac{\epsilon}{4k(B+1)}, i=0, 1, \dots, k-1.$$
 	We have that 
 	\begin{align}|u_0|+|u_{1}|+\dots+|u_{k-1}|\leq \frac{\epsilon}{B+1}\end{align}
 and (see \eqref{theta}), 
 \begin{align}
 u_{4kn+i}\leq u_i,\end{align} $i=0,1,\dots,k-1$,   for all $n$ if $B\geq 0$.\\
  This implies that for  $|u_{4kn+i}|\leq |u_i|\leq \frac{\epsilon}{4k(B+1)}\leq \epsilon$, we have found $\delta= {\epsilon}/{(B+1)}$ such that $|u_0|+|u_{1}|+\dots|u_{k-1}|\leq \delta$. Thus, the zero equilibrium point is locally stable. On the other hand (see Theorem 4.2), $x_n$ tends to zero as $n$ goes to infinity. The zero equilibrium being a global attractor and locally stable, it is globally asymptotically stable.   
 	\end{proof}
 		\begin{theorem}
 		 	Assume $A\neq 1$. The zero equilibrium point of \eqref{eq1cstp} is asymptotically stable for $|A|> 1$ and unstable for $|A|< 1$. Furthermore, all non zero equilibrium points of \eqref{eq1cstp} are non-hyperbolic.   
 		\end{theorem}
 		\begin{proof}
 		The equilibrium points of \eqref{eq1cstp} satisfy $u(A+Bu^k-1)=0$. Let {\scriptsize  \begin{align}
 		f(u_n, u_{n+4}, \dots, u_{n+4(k-1)})=	u_{n+4k}=\frac{u_n}{A+B\prod\limits_{i=1}^{k}
 			u_{n+4(i-1)}}.\end{align}} We have
{\scriptsize 	\begin{align}
 		f_{,u_n}=\frac{A}{(A+B\prod\limits_{i=1}^{k}
 		u_{n+4(i-1)})^2}, \quad
 	f_{,u_{n+4j}}=\frac{-Bu_n^2}{(A+B\prod\limits_{i=1}^{k}
 		u_{n+4(i-1)})^2} \prod\limits_{\substack{ i=1\\i\neq j}}^{k-1}
 	u_{n+4i},\quad j=1,2,\dots, k-1.
 	\end{align}}
  \begin{itemize}
 	\item For the equilibrium point $u=0$, we have $	f_{,u_n}(0,0,\dots,0)=1/A$ and $f_{,u_{n+4j}}(0,0,\dots,0)=0, j=1,2,\dots, k-1$. Thus, the characteristic equation associated with \eqref{eq1cstp} is $\lambda ^{4k}-\frac{1}{A}=0.$ Therefore, $|\lambda |< 1$ if $|A|> 1$ (that is, locally asymptotically stable) and $|\lambda |> 1$ if $|A|< 1$ (that is, unstable).
 		\item The non-zero equilibrium points $u$ satisfy $A+Bu^k-1=0$. Then $f_{,u_n}(u,u,\dots,u)=A$ and $f_{,u_{n+4j}}(u,u,\dots,u)=A-1, j=1,2,\dots, k-1$. Thus, the characteristic equation associated with \eqref{eq1cstp} is 
 		\begin{align}\label{charaanot1}
 		\lambda ^{4k}-(A-1)\lambda ^{4k-4}-\dots-(A-1)\lambda ^4 -A=0.
 		\end{align}
 Multiplying the above equation by $1-\lambda ^4$, we get (after simplification) \begin{align}\label{chara}
 (1-\lambda ^{4k})(\lambda ^4-A)=0.\end{align}		
It follows that, for $A< 0$, the solutions of \eqref{chara} are $\lambda_r=-Ae^{i\frac{i(2r+1)\pi}{4}},\; r=0,1,2,3$ or $\lambda_r =e^{i\frac{2r\pi}{4k}}, p= 1,2,\dots, k-1, k+1, \dots, 2k-1,2k+1,\dots,2k-1, 2k+1,\dots,3k-1,3k+1,\dots, 4k-1.$ 	For $A>0$, the solutions of \eqref{chara} are $\lambda_r=Ae^{i\frac{2r\pi}{4}},\; r=0,1,2,3$ or $\lambda_r =e^{i\frac{2r\pi}{4k}}, p= 1,2,\dots, k-1, k+1, \dots, 2k-1,2k+1,\dots,2k-1, 2k+1,\dots,3k-1,3k+1,\dots, 4k-1.$  		
 Therefore, for $k>1$, there exists a root of \eqref{charaanot1} with modulus equal to one.
 \end{itemize}
 		\end{proof}
\section{Conclusion}
We studied the difference equation $\eta _{n+1}={\eta_{n-4k+1}}/{(a_n +b_n \prod_{i=1}^{k}
	\eta_{n-4i+1})}$ by performing its symmetry analysis and we used the canonical coordinate to obtain its invariants. These invariants are utilized to derive the solutions in closed form. We demonstrated that all the formula solutions in \cite{alma} are special cases of our findings. Some conditions for existence of $4k$ and $8k$ periodic solutions were established. Finally, we investigated the stability of the solution of the difference equation and proved the existence of non-hyperbolic and globally asymptotically stable equilibrium points. 

\end{document}